\documentclass[11pt,oneside]{amsart}

\title{The classification of countable models of set theory}
\author{John Clemens}
\address{John Clemens, Boise State University, 1910 University Dr, Boise, ID 83725}
\email{johnclemens@boisestate.edu}
\author{Samuel Coskey}
\address{Samuel Coskey, Boise State University, 1910 University Dr, Boise, ID 83725}
\email{scoskey@gmail.com}
\author{Samuel Dworetzky}
\address{Samuel Dworetzky, University of Denver, 2390 S York St, Denver, CO 80208}
\email{sam.dworetzky@du.edu}

\usepackage[marginratio=1:1]{geometry}
\usepackage{setspace}
\usepackage{mathpazo,bm,amssymb}

\newtheorem{thm}{Theorem}[section]
\newtheorem{lem}[thm]{Lemma}
\newtheorem{prop}[thm]{Proposition}

\theoremstyle{definition}
\newtheorem{defn}[thm]{Definition}
\theoremstyle{remark}

\renewcommand{\setminus}{\smallsetminus}

\newcommand{\WO}{\ensuremath{\mathrm{WO}}}
\newcommand{\PA}{\ensuremath{\mathrm{PA}}}
\newcommand{\ZF}{\ensuremath{\mathrm{ZF}}}
\newcommand{\ZFC}{\ensuremath{\mathrm{ZFC}}}
\newcommand{\ZFGC}{\ensuremath{\mathrm{ZFGC}}}
\newcommand{\WFT}{\ensuremath{\mathrm{WFT}}}
\newcommand{\Diag}{\ensuremath{\mathrm{Diag}}}

\DeclareMathOperator{\gap}{\mathop{\mathrm{gap}}}

\DeclareMathOperator{\Th}{\mathop{\mathrm{Th}}}

\usepackage{etoolbox}
\makeatletter\pretocmd{\@seccntformat}{\S}{}{}
  \pretocmd{\@subseccntformat}{\S}{}{}\makeatother

\begin{document}
\begin{abstract}
  We study the complexity of the classification problem for countable models of set theory (\ZFC). We prove that the classification of arbitrary countable models of \ZFC\ is Borel complete, meaning that it is as complex as it can conceivably be. We then give partial results concerning the classification of countable well-founded models of \ZFC.
\end{abstract}
\maketitle

\section{Introduction}

In set theory we have a number of fundamental methods to construct models of \ZFC: ultrapower constructions, forcing constructions, model-theoretic constructions using compactness, and so on. With such powerful and versatile methods of building models, it is natural to expect that the classification of models of \ZFC\ is a very complex problem. In this article we examine the classification problem for countable models of \ZFC\ from the point of view of Borel complexity theory, which we will describe shortly.

Our first result will be to confirm the above intuition and show that, assuming \ZFC\ has any models, the classification of countable models of \ZFC\ is ``Borel complete''. This level of complexity will be defined below, but for the moment we note that it is the maximum conceivable complexity for this problem. Stronger, we will show that for any consistent theory $T$ extending \ZFC, the classification of countable models of $T$ is Borel complete.

The proof of this fact will make use of the close analogy between models of \ZFC\ and models of \PA, together with the fact that the analogous result has already been established for countable models of \PA\ in \cite{coskey-kossak}. In that article, the authors used a construction due to Gaifman called a ``canonical $I$-model'' to establish that for any completion $T$ of \PA, the classification of countable models of $T$ is Borel complete. In the present article, we will show how Gaifman's construction may be used to build models of \ZFC, and how the argument of \cite{coskey-kossak} thus gives the desired conclusion for models of \ZFC.
 

Of course, Gaifman's construction produces \emph{nonstandard} (meaning ill-founded) models of \PA. Our modified construction produces nonstandard models of \ZFC\ as well. Thus it is natural to ask what is the complexity of the classification of countable \emph{standard} (meaning well-founded) models of \ZFC. Here the answer must be somewhat more subtle than before since, for instance, the complexity of countable standard models of $T$ will depend on the particular completion $T$ of \ZFC\ that one studies. Even the number of countable standard models depends on $T$. In fact, Enayat has shown in \cite{enayat-counting} that the number of countable standard models of $T$ up to isomorphism may be any cardinal $\leq\aleph_1$ or continuum.

While we do not identify the precise complexity of the classification of countable standard models, we will provide several partial results on the subject. For instance, we show that the complexity of the classification of standard models of \ZFC\ lies somewhat below the level of a Borel complete classification problem. Additionally, for several particular completions $T$ of \ZFC, we identify bounds on the complexity of the classification of countable standard models of $T$.

In order to discuss these results formally, we will need to describe the Borel complexity theory of classification problems. First, if $\mathcal L$ is any countable first-order relational language then we may form the standard Borel space of all countable $\mathcal L$-structures:
\[X_{\mathcal L}=\prod_{R\in\mathcal{L}}2^{{\omega}^{{a(R)}}}\,,
\]
where $a(R)$ denotes the arity of the logical symbol $R$. If $T$ is any $\mathcal L$-theory we study the Borel subset consisting of just the models of $T$:
\[X_T=\lbrace M\in X_{\mathcal L}\mid M\models T\rbrace .
\]
We then identify the classification problem for countable models of $T$ with the isomorphism equivalence relation $\cong_{T}$ on $X_{T}$.

In order to compare the complexity of two classification problems, we use the notion of Borel reducibility. Generally, if $X,Y$ are standard Borel spaces and $E,F$ are equivalence relations on $X,Y$ respectively, then we say $E$ is \emph{Borel reducible} to $F$ (denoted $E\leq_{B}F$) if there is a Borel function $f\colon X\to Y$ such that
\[x\mathrel{E}x' \iff f(x)\mathrel{F}f(x')\,.
\]
Intuitively, if $E$ is Borel reducible to $F$, then we say that the classification problem for elements of $Y$ up to $F$-equivalence is at least as complex as the classification problem for elements of $X$ up to $E$-equivalence.

The study of Borel reducibility has provided a series of benchmark equivalence relations with which to compare a given classification problem. One of the simplest equivalence relations is the equality relation $=$ on $2^\omega$. By the Silver dichotomy, $=$ is the minimum among all Borel equivalence relations with uncountably many equivalence classes. Just above $=$ is the almost equality relation $E_0$ on $2^\omega$ defined by $x\mathrel{E}_0x'$ iff $x(n)=x'(n)$ for all but finitely many $n$. By the Glimm--Effros dichotomy \cite{hkl}, any Borel equivalence relation is either Borel reducible to $=$ or else $E_0$ is Borel reducible to it.

At the higher end of the complexity spectrum, there is a maximum possible complexity among isomorphism classification problems for classes of countable structures. First, we say that an equivalence relation $E$ is \emph{classifiable by countable structures} if $E$ is Borel reducible to $\cong_T$ for some theory $T$. Then, we say that $E$ is \emph{Borel complete} if for any first-order theory $T$ there is a Borel reduction from $\cong_{T}$ to $E$. Some well-known examples of Borel complete classifications include the isomorphism equivalence relations on countable connected graphs and on countable linear orders.

In the next section, we review the argument that the isomorphism relation on the class of countable models of \PA\ is Borel complete. We then show how to modify the details of the argument to show that the isomorphism relation on countable models of \ZFC\ is Borel complete too. In the third section, we study just the standard models of \ZFC. We show in an appropriate sense that the classification of countable standard models is strictly lower than Borel complete. We also show that if $T$ is a theory of Cohen forcing models, then the classification of countable standard models of $T$ lies at or above $E_0$ in complexity. Finally, we show that under a mild hypothesis, if $T$ is the theory of $L_{\omega_1}$ then the classification of countable models of $T$ is not Borel reducible to a Borel equivalence relation.

\textbf{Acknowledgement.} This work represents a portion of the third author's master's thesis \cite{sam-thesis}. The thesis was written at Boise State University under the supervision of the second author, with significant input from the first author. We would like to thank Ali Enayat and Iian Smythe for a number of helpful communications on the subject.

\section{Arbitrary models of \ZFC}

In this section we show that the classification of countable models of \ZFC\ is Borel complete by adapting the proof of the \PA\ case from \cite{coskey-kossak}. We begin by reviewing the key elements of the proof of the \PA\ version.

As we have said, the argument relies on the details of a construction due to Gaifman called a canonical $I$-model of \PA. The construction begins with the following definition. Let $M\models\PA$ and let $p(v)$ be a type (of arity $1$) over $M$. Then $p$ is said to be \emph{minimal} if it is:
\begin{itemize}
  \item unbounded: for all $a\in M$ we have $(a<v)\in p(v)$; and
  \item indiscernible: if $M\prec N$ and $a_{1}<\cdots<a_{n}$ and $b_{1}<\cdots<b_{n}$ are two sequences of realizations of $p(v)$ in $N$, then $N(\overline{a})\equiv N(\overline{b})$.
\end{itemize}
Gaifman showed that every model $M\models\PA$ admits a minimal type.

Next if $M\models\PA$ and $I$ is a given linear ordering, the \emph{canonical $I$-model} $M(I)$, constructed with respect to some fixed minimal type $p$ over $M$, is generated by $M$ together with an $I$-ordered sequence of realizations of $p$. Canonical $I$-models have many useful properties, but for our purposes it is enough to know the following two facts:
\begin{enumerate}
  \item The realizations of $p$ in $M(I)$ form a sequence of order indiscernibles; and
  \item The ordertype $I$ can be recovered from any isomorphic copy of $M(I)$. (Specifically $I$ will be the ordertype of the set of \emph{gaps} in $M(I)$, which we will define and see below.)
\end{enumerate}
It follows from property~(a) that $I\cong I'$ implies $M(I)\cong M(I')$, and from property~(b) that $M(I)\cong M(I')$ implies $I\cong I'$. Thus Coskey--Kossak were able to conclude that there exists a Borel reduction from the class of countable linear orders to the class of countable models of \PA\ which carries $I\mapsto M(I)$.

We now aim to adapt this construction to the case of models of set theory. We begin with the appropriate analog of the notion of a minimal type. First, if $M\models\ZFC$ and $p(v)$ is a type with parameters from $M$, we will say $p(v)$ is an \emph{$\omega$-type} over $M$ if $(v<\omega)\in p(v)$.

\begin{defn}
  Let $M\models\ZFC$ and let $p(v)$ be an $\omega$-type over $M$. We say that $p(v)$ is \emph{$\omega$-minimal} if it is:
  \begin{itemize}
  \item $\omega$-unbounded: for all $\alpha\in\omega^M$ we have $(\alpha<v)\in p(v)$; and
  \item indiscernible: if $M\prec N$ and $\alpha_{0}<\cdots<\alpha_{n}$ and $\beta_{0}<\cdots<\beta_{n}$ are two sequences of realizations of $p(v)$ in $N$ then $N(\bar\alpha)\equiv N(\bar\beta)$.
  \end{itemize}
\end{defn}

The following is the analog of Gaifman's theorem, and the proof is a straightforward adaptation of the classical version. Here we summarize \cite[Thoerem~3.1.2]{kossak-schmerl}; alternatively see \cite{gaifman} or \cite[Proposition~11.4]{wong}.

\begin{prop}
  For any $M\models\ZFC$, $M$ admits an $\omega$-minimal type.
\end{prop}

\begin{proof}
  Let $\varphi_i(\bar x)$ be an enumeration of the formulas. We inductively construct a sequence of formulas $\theta_i(v)$ satisfying:
  \begin{enumerate}
    \item $\theta_{i+1}(v)\to\theta_i(v)$;
    \item for all $\alpha\in\omega^M$ there exists $\beta\in\omega^M$ such that $\alpha<\beta$ and $\theta_{i+1}(\beta)$, and;
    \item $\theta_i(v)$ ``settles'' $\varphi_i$ in the sense that $M$ satisfies for all increasing $\bar x\in\omega$ we have $\bigwedge_j\theta_i(x_j)\to\varphi_i(\bar x)$, or else for all increasing $\bar x\in\omega$ we have $\bigwedge_j\theta_i(x_j)\to\neg\varphi_i(\bar x)$.
  \end{enumerate}
  To do so we use Ramsey's theorem, as formalized in \ZFC, repeatedly. That is, given $\theta_i(v)$, regard it as an unbounded subset of $\omega^M$. The formalized Ramsey theorem then implies it has an unbounded subset, definable by some $\theta_{i+1}(v)$, which is homogeneous for the partition determined by $\varphi_i(\bar x)$.
  
  Now let $p(v)$ be the deductive closure of the $\theta_i(v)$ and $\alpha<v$ for $\alpha\in\omega^M$. Then $p(v)$ is consistent and $\omega$-unbounded by (a),(b), and it is not difficult to confirm that $p(v)$ is indiscernible by (c). Thus $p(v)$ is $\omega$-minimal.
\end{proof}

We will also need the fact that $\omega$-minimal types are \emph{$\omega$-strongly definable}: for every formula $\varphi(v,z)\in\mathcal{L}_{M}$ there exists a formula $\theta(v)\in p(v)$ such that
\[M\models \forall z\in\omega \left[
    \forall^{\infty}v\in\omega(\theta(v)\rightarrow\varphi(v,z))
    \vee
    \forall^{\infty}v\in\omega(\theta(v)\rightarrow\neg\varphi(v,z))
  \right]
\]
Here, $\forall^\infty v\in\omega$ means ``for all $v$ outside a bounded subset of $\omega$.'' Once again, the proof is a straightforward adaptation of the classical version, we follow Exercise~3.6.5 and Lemma~3.1.13 of \cite{kossak-schmerl}.

\begin{prop}
  If $M\models\ZFC$ and $p(v)$ is an $\omega$-minimal type over $M$, then $p(v)$ is $\omega$-strongly definable.
\end{prop}

\begin{proof}
  Let $\varphi(v,z)$ be given. By indiscernibility, we can find a formula $\theta(v)\in p(v)$ such that
  \[M\models\forall x\forall y\forall v\,[
    (\theta(x)\wedge\theta(y)\wedge\theta(v)\wedge x<y<v)
    \to
    (\forall z\leq x(\varphi(y,z)\leftrightarrow\phi(v,z)))
  ]
  \]
  Now given $z\in\omega^M$, suppose that $M\models\neg\forall^\infty v\in\omega(\theta(v)\to\phi(v,z))$. Then we can find $x,y\in\omega^M$ such that $z\leq x\leq y$, and $\theta(x)$, $\theta(y)$, $\neg\phi(y,z)$ are true in $M$. By the choice of $\theta$, if $v\in\omega^M$ is such that $y<v$ and $\theta(v)$ is true in $M$, then $\neg\phi(v,z)$ is true in $M$ too. Thus we have $M\models\forall^\infty v\in\omega(\theta(v)\to\neg\phi(v,z))$, as desired.
\end{proof}

%
%


In order to construct the models $M(I)$, we will assume $M$ is a model of \ZFGC, that is, \ZF\ together with the global choice axiom. This means $M$ is a structure in the expanded language with an additional function symbol $F$, and $F$ is interpreted as a function with the property that for all nonempty $x\in M$ we have $M\models F(x)\in x$. The global choice axiom helps us mimic the \PA\ arguments because the theory \ZFGC\ has built-in Skolem functions.

The next definition, which we promised earlier, is the key to recovering the order type of $I$ from the isomorphism type of $M(I)$.

\begin{defn}
  Let $M\models\ZFGC$, and let $\beta\in\omega^M$. We define the following sets:
  \begin{itemize}
  \item Let $M_{\omega}(\beta) = \lbrace \alpha\in \omega^{M}:~\text{for some Skolem function $t$},~t(\beta)\in \omega^{M}~\wedge~\alpha<t(\beta)\rbrace$.
  \item Let $M_{\omega}[\beta] = \lbrace \alpha\in \omega^{M}:~\text{for any Skolem function $t$},~\text{if}~~t(\alpha)\in \omega^{M}~\text{then}~t(\alpha)<\beta \rbrace$.
  \end{itemize}
  We then define the \emph{$\omega$-gap} of $\beta$ as $\gap_{\omega}(\beta)=M_{\omega}(\beta)\setminus M_{\omega}[\beta]$.
\end{defn}

The following result shows how minimal types and gaps are related; see also \cite[Lemma~3.1.18]{kossak-schmerl}.

\begin{prop}
  \label{prop:rare}
  Let $M\models\ZFGC$ and let $p(v)$ be an $\omega$-minimal type over $M$. Then $p(v)$ is \emph{rare}, which means that if $M\prec N$ and $\alpha,\beta\in\omega^N$ are distinct witnesses of $p(v)$, then $\alpha$ and $\beta$ lie in distinct $\omega$-gaps.
\end{prop}

\begin{proof}
  Assume $\alpha<\beta$, and let $N'$ be an elementary extension of $N$ with some $\gamma\in\omega^{N'}$ such that $\beta$ lies below $\gap_\omega(\gamma)$. Then for any Skolem function $t$, we have $t(\beta)<\gamma$. Since $p(v)$ is indiscernible, we have $t(\alpha)<\beta$ too. Thus $\alpha,\beta$ lie in distinct $\omega$-gaps.
\end{proof}

Before we construct the models $M(I)$ along a linear order $I$, we first consider the case of adjoining a single new witness for $p$ to a model $M$. Let $M\models\ZFGC$ and let $p(v)$ be an $\omega$-minimal type over $M$. In the following result we will let $M(\{\gamma\})$ denote the elementary extension of $M$ obtained by adjoining a single witness $\gamma$ for $p(v)$. That is, $M(\{\gamma\})$ is the prime model of the elementary diagram of $M$ together with the sentences $p(\gamma)$. The prime model exists thanks to the built-in Skolem functions of $\ZFGC$.

The following proposition is a straightforward analog of \cite[Proposition~4.8]{gaifman}; we follow the proof in \cite[Proposition~10.4]{wong}.
  
\begin{lem}
  \label{lem:gaifman1step}
  Let $M\models\ZFGC$, $p(v)$ be an $\omega$-minimal type over $M$, and let $M(\{\gamma\})$ be as above. Then, $\omega^{M(\{\gamma\})}=\omega^{M}\cup\gap_{\omega}(\gamma)$.
\end{lem}

\begin{proof}
   Let $\beta\in\omega^{M(\{\gamma\})}$, and write $\beta=t(\gamma)$ for some Skolem term $t$. By $\omega$-strong definability, we can find $\theta(v)\in p(v)$ such that
  \[M\models\forall z\in\omega \left[
    (\forall^\infty v\in\omega(\theta(v)\rightarrow z=t(v)))
    \vee
    (\forall^\infty v\in\omega(\theta(v)\rightarrow z\neq t(v)))
  \right]
  \]
  First suppose $M$ satisfies $\forall z\in\omega(\forall^\infty v\in\omega(\theta(v)\rightarrow z\neq t(v)))$. Then $M(\{\gamma\})$ satisfies the same sentence. Let $s(z)$ be the least $v$ such that $\theta(v)\rightarrow z\neq t(v)$. Since it is true in $M(\{\gamma\})$ that $\theta(\gamma)\rightarrow \beta=t(\gamma)$, we must have that $s(\beta)>\gamma$. Thus $s$ is a Skolem function witnessing that $\beta\in\gap_\omega(\gamma)$.
  
  Next suppose that $M$ satisfies $\exists z\in\omega(\forall^\infty v\in\omega(\theta(v)\rightarrow z=t(v)))$. Then we can find $m_0,m\in M$ such that $M$ satisfies $\forall v(v\geq m_0\wedge\theta(v)\rightarrow m=t(v))$. It follows that $M(\{\gamma\})$ satisfies the same sentence, and we conclude that $\beta=t(\gamma)=m\in M$, completing the proof.
\end{proof}

We remark that the lemma implies $M(\{\gamma\})$ is an \emph{$\omega$-end extension} of $M$, meaning for any $\alpha\in\omega^M$ and any $\beta\in\omega^{M(\{\gamma\})}\setminus\omega^M$ we have $\alpha<\beta$.
  
The following result describes the construction of the model $M(I)$. It also asserts the key property which will allow us to recover the ordertype of $I$ from the isomorphism type of the model $M(I)$.

\begin{thm}
  \label{thm:gaifman}
  Let $M\models\ZFGC$ and $p(v)$ be an $\omega$-minimal type over $M$. Let $(I,<)$ be a linearly ordered set. Then there is an $\omega$-end extension $M\prec N$ generated over $M$ by a set $X = \lbrace \alpha_i\mid i\in I\rbrace\subset N$ such that $\alpha_i<\alpha_j$ for all $i<j$ and $\omega^N = \omega^{M}\cup\bigcup_{i\in I}gap_{\omega}(\alpha_{i})$. 
\end{thm}
      
\begin{proof}
  We first construct the extension $N$. We form the theory
  \[T=\Diag_{el}(M)\cup\bigcup_{i\in I}p(\alpha_{i})\cup\lbrace\alpha_{i}\in\alpha_{j}\mid i<j\wedge i,j\in I\rbrace
  \]
  where each $\alpha_{i}$ is a new constant symbol. We then let $N$ be the prime model of $T$, that is, the Skolem hull of $X=\{\alpha_i\mid i\in I\}$ in any model of $T$. By the argument of Lemma~\ref{lem:gaifman1step}, we have that $N$ is an $\omega$-end extension of $M$.
  
  It remains to show that $\omega^N = \omega^{M}\cup \bigcup_{i\in I}\gap_{\omega}(\alpha_{i})$. For this, let $\beta\in\omega^N$. Since $N$ is a Skolem hull, we can find a formula $\eta$ in the language of set theory and $\alpha_{i_1},...,\alpha_{i_n}\in X$ such that $\beta=\eta(\alpha_{i_1},...,\alpha_{i_n})$. Now, let $N_0$ denote the Skolem hull of $M\cup\lbrace \alpha_{i_1},\ldots,\alpha_{i_n}\rbrace$. By Proposition~\ref{prop:rare}, the gaps of $\alpha_{i_0},\ldots\alpha_{i_n}$ are disjoint. So, we have that $\gap_{N_{0}}(\alpha_{i_1})<\cdots<\gap_{N_{0}}(\alpha_{i_n})$. Using Lemma~\ref{lem:gaifman1step} inductively, we conclude that $\omega^{N_0}= \omega^{M}\cup\bigcup_{j\leq n}\gap_{N_{0}}(\alpha_{i_j})$.

  Now, it follows that $\beta$ is an element of $\omega^{M}$ or one of the gaps $\gap_{N_{0}}(\alpha_{i_j})$ for some $j\leq n$. To finish the proof, we note that $N_{0}$ and $N$ have the same Skolem functions. So, we conclude that $\beta\in \omega^{M}$ or $\beta\in \gap_{N}(\alpha_{i})$ for some $i\in I$. Thus $\omega^N = \omega^{M}\cup\bigcup_{i\in I}\gap_{\omega}(\alpha_{i})$. 
\end{proof}

As in the \PA\ case, we will use $M(I)$ to denote elementary extension $N$ constructed in Theorem~\ref{thm:gaifman}. We now use the construction of $M(I)$ to obtain a Borel reduction from the isomorphism relation on the class of countable linear orders to the isomorphism relation on the class of countable models of set theory. In particular, this will show that the isomorphism relation on the class of countable models of set theory is Borel complete.

\begin{thm}
  \label{ZFCBC}
  Let $T$ be a consistent completion of \ZFGC. Then isomorphism relation on countable models of $T$ is Borel complete.
\end{thm}

\begin{proof}
  Any consistent completion of \ZFGC\ has a prime model $M$. We need to show that:
  \begin{enumerate}
  \item The construction of $M(I)$ is Borel;
  \item $I\cong I'\Rightarrow M(I)\cong M(I')$, and;
  \item $M(I)\cong M(I')\Rightarrow I\cong I'$.
  \end{enumerate}
  For item (a), we observe that the construction of $M(I)$ can be carried out as a Henkin construction followed by taking a Skolem hull. It is not difficult to see that both of these procedures may be carried out in a Borel fashion.

  For item (b) we note that the generating set $\{\alpha_i\}$ of $M(I)$ over $M$ is a set of order indiscernibles. It is a well-known property of order indiscernibles that order isomorphisms between sequences of order-indiscernibles extend to isomorphisms between the models they generate (see for instance \cite[Lemma~5.2.6]{Marker}).

  Finally, item (c) follows from the gap information provided in Theorem~\ref{thm:gaifman}. To begin, note that an isomorphism $\alpha\colon M(I)\cong M(I')$ induces an order-preserving isomorphism from the set of $\omega$-gaps of $M(I)$ to the set of $\omega$-gaps of $M(I')$. Since $\omega$-minimal types are rare, we know that there is just one witness for $p(v)$ in each nontrivial $\omega$-gap of $M(I)$ or $M(I')$. Since the witnesses of $p(v)$ are of ordertypes $I$ and $I'$ respectively, $\alpha$ induces an order-preserving isomorphism $I\cong I'$.

  We have thus established that there is a Borel reduction $\cong_{LO}~\leq_{B}~\cong_T$, and in particular that $\cong_T$ is Borel complete. 
\end{proof}

Since the consistency of \ZFC\ implies the consistency of \ZFGC, it is a consequence of the theorem that if \ZFC\ is consistent then the classification of all models of \ZFC\ is Borel complete. Of course, it is natural to ask whether Theorem~\ref{ZFCBC} holds for an arbitrary completion $T$ of \ZFC\ which does not necessarily entail global choice. We can certainly say that there are other hypotheses on $T$ which will suffice. For example if $T$ has a prime model or a model with just a finite number of $\omega$-gaps, then the above proof will go through with minor modifications.




\section{Well-founded models of \ZFC}

In this section we study the classification of well-founded models of \ZFC. If $T$ is a completion of \ZFC\ and $T$ possesses well-founded models, we let $\WFT$ denote the set of codes for well-founded models of $T$, and $\cong_\WFT$ denote the isomorphism relation restricted to $\WFT$.

We remark that $\WFT$ is not a Borel subset of the space of countable models of $T$, and so we must be careful how we study $\cong_\WFT$ with respect to Borel reducibility. While the domain of a Borel reduction function should always be a standard Borel space, the range may be contained in any subset such as $\WFT$. This means it still makes good sense to ask questions about lower bounds. For instance we can ask whether $\cong_\WFT$ is Borel complete in the sense that some Borel complete equivalence relation is Borel reducible to it. On the other hand, in order to ask questions about upper bounds it is usual to use a somewhat broader class of reduction functions than just the Borel reductions. We will use the absolutely $\bm{\Delta}^1_2$ functions, described below.

Our first result establishes that the classification of well-founded countable models of set theory is properly less complex than the classification of arbitrary countable models.

\begin{prop}
  If $T$ is any completion of \ZFC, then $\cong_\WFT$ is not Borel complete.
\end{prop}

\begin{proof}
  We first note that the set $\WFT$ of well-founded countable models of $T$ is a $\bm{\Pi}^1_1$ set, with rank function inherited from the usual rank function for well-founded binary relations. In fact, the rank function is simply $M\mapsto o(M)$, the ordertype of the ordinals of $M$.

  Now suppose towards a contradiction that $\cong_\WFT$ is Borel complete. Then there is, for instance, a Borel reduction $f$ from the isomorphism relation $\cong$ on the set $2^{\omega\times\omega}$ of all countable binary relations to $\cong_\WFT$. The range $f(X)$ is a $\bm{\Sigma}^1_1$ subset of $\WFT$. By the boundedness theorem \cite[Theorem~31.2]{kechris}, it follows that the rank function restricted to $f(X)$ is bounded by some ordinal $\alpha$.

  The set $\WFT_\alpha$ of models of $T$ of rank bounded by $\alpha$ is a Borel set, and we claim the isomorphism relation on $\WFT_\alpha$ is Borel reducible to the isomorphism relation on codes for countable well-founded trees of rank $\alpha$. For this, given an element $M\in\WFT_\alpha$ we can produce in a Borel way a code for a tree $T_M$ which represents the model $M$ in a standard way. Thus the root node of $T_M$ represents $M$ itself, the children of the root represent the elements of $M$, and so on, and all leaves of $T_M$ represent the empty set. The tree $T_M$ has the same ordinal rank as that of $M$. Moreover, models $M$ and $M'$ are isomorphic if and only if the codes for the corresponding trees $T_M$ and $T_{M'}$ are isomorphic. This establishes the claim.
  
  Now it is well-known that the isomorphism relation on well-founded trees of any fixed countable rank is Borel (these equivalence relations are studied in \cite{friedman}). It follows from the claim that the isomorphism relation on $\WFT_\alpha$ is Borel as well. Thus we conclude that the Borel complete equivalence relation $\cong$ is Borel reducible to a Borel equivalence relation. But this contradicts the well-known fact from \cite{friedman} that any Borel complete equivalence relation is not itself Borel.
\end{proof}

In the article \cite{enayat-counting}, the author shows that the number of isomorphism equivalence classes in $\WFT$ can have several values, such as $0$, finite, countable, $\aleph_1$, and continuum. In the rest of this section we consider the question of what is the Borel complexity of $\cong_\WFT$ for several special theories $T$.

Recall that $E_{0}$ denotes the equivalence relation defined on $2^{\omega}$ by $x\mathrel{E_0}x'$ if and only if $x(n)=x'(n)$ for all but finitely many $n$. As stated in the introduction, the Glimm-Effros dichotomy states that for any Borel equivalence relation $E$, either $E$ is smooth (Borel reducible to $=$) or else $E_0$ is Borel reducible to $E$.

\begin{thm}
  Assume $M$ is a countable well-founded model of \ZFC, let $g$ be Cohen generic over $M$, and let $T=\Th(M[g])$. Then there is a Borel reduction from $E_0$ to $\cong_\WFT$.
\end{thm}

\begin{proof}
  Let $X\subset 2^\omega$ be the set of reals of $V$ which are Cohen generic over $M$. Define the equivalence relation $E$ on $X$ by
  \[g_1\mathrel{E}g_2\Leftrightarrow M[g_1] = M[g_2].
  \]
  Since the forcing relation is definable in $M$, and since $g_1\mathrel{E}g_2$ iff $g_1\in M[g_2]$ and $g_2\in M[g_1]$, one can conclude that $E$ is arithmetic as a set of pairs and in particular $E$ is a Borel equivalence relation. (Alternatively, see \cite[Theorem~3.5.1]{grigorieff}.) In fact $E$ is a countable Borel equivalence relation, meaning each of its equivalence classes is countable.



  We first show that $E_0\leq_B E$. For this, if $g_1,g_2\in X$ and $g_1\mathrel{E}_0g_2$, then $g_1$ and $g_2$ are definable from one another and it follows that $M[g_1]=M[g_2]$. This implies that the restriction $E_0\restriction X$ is a subrelation of $E$. Using some basic facts about $E_0$ and countable Borel equivalence relations (note that $X$ is comeager and see \cite[Propositions~6.1.9,~6.1.10]{gao}), we can conclude that $E$ is not smooth. It then follows from the Glimm--Effros dichotomy that $E_0\leq_B E$.
  
  Next we show that $E\leq_B\mathord{\cong}_{WFT}$. Consider the map $g\mapsto x$ carrying a Cohen generic real $g$ over $M$ to a code $x\in2^{\omega\times\omega}$ for $M[g]$. The mapping is Borel; here we use a code for $M$ as a parameter, together with the definability of the forcing relation. Clearly we have $g\mathrel{E}g'\implies M[g]\cong M[g']$; conversely if $M[g]\cong M[g']$, then since the structures are transitive, we have $M[g]=M[g']$ and so $g\mathrel{E}g'$. Thus we have shown $E\leq_B\mathord{\cong}_{WFT}$.
  
  Putting the results of the last two paragraphs together, we conclude that $E_0\leq_B\mathord{\cong}_{WFT}$.
\end{proof}

We now turn to the study of a second theory $T$. In this case $\cong_\WFT$ will be compared with the equivalence relation $E_{\omega_1}$. The relation $E_{\omega_1}$ is equivalence of codes for countable ordinals, that is, the isomorphism equivalence relation on the set of countable well-ordered relations. The domains of both $\cong_\WFT$ and $E_{\omega_1}$ are non-Borel sets, so we shall need to compare them with respect to absolutely $\bm{\Delta}^1_2$ reduction functions. Here a function is \emph{absolutely~$\bm{\Delta}^1_2$} if it possesses $\bm{\Sigma}^1_2$ and $\bm{\Pi}^1_2$ definitions which are equivalent in all forcing extensions.

\begin{thm}
  Assume $0^\sharp$ exists, and let $T=\Th(L_{\omega_1})$. Then $T$ is a completion of \ZFC, and there exists an absolutely $\bm{\Delta}^1_2$ reduction from $E_{\omega_1}$ to $\cong_\WFT$.
\end{thm}

\begin{proof}
  By \cite[Corollary~18.3]{jech}, the existence of $0^\sharp$ implies that $\omega_1^V$ is inaccessible in $L$. It follows that $T$ is a completion of \ZFC.
  
  For the reduction we first show that there is a continuous mapping $g$ with the property that if $x$ is a code for a countable ordinal $\alpha$, then $g(x)$ is a code for a countable ordinal $\beta$ such that $\alpha\leq\beta$ and $L_\beta\models T$. In order to do so, let $G$ be the game in which Players~I and~II alternate playing digits to construct $x,y\in 2^\omega$. Player~II wins if either $x\notin\WO$, or; $x,y$ are codes for ordinals $\alpha,\beta$, $\alpha\leq\beta$, and $L_\beta\models T$. We claim that Player~II has a winning strategy for $G$. Admitting this claim, we let $g$ be the continuous mapping which takes a play $x$ of Player~I to the corresponding play $y$ of Player~II according to the strategy.
  
   To establish that Player~II has a winning strategy, first observe that the winning condition for $G$ is a Boolean combination of lightface analytic sets. It follows from the existence of $0^\sharp$ together with a result of Martin \cite[Theorem~31.4]{kanamori} that $G$ is determined. Hence it is enough to show that Player~I does not have a winning strategy for $G$. To see this, first note that by a simple reflection argument there are unboundedly many $\beta<\omega_1$ such that $L_\beta\models T$. Now suppose Player~I does have a winning strategy for $G$ and let $S\subset\WO$ be the set of all reals $x$ constructed according to the strategy. Then $S$ is a $\bm{\Sigma}_1^1$ subset of $\WO$ and so the boundedness theorem \cite[Theorem~31.2]{kechris} implies $S$ is bounded in $\omega_1$. This is a contradiction, since Player~II can now defeat the strategy by playing a code $y$ for some $\beta$ above $S$ such that $L_\beta\models T$.

 
  Next we will show that there exists an absolutely $\bm{\Delta}^1_2$ function $f$ such that if $x$ is a code for an ordinal $\alpha$, then $f(x)$ is a code for $L_\beta$, where $\beta$ is the $\alpha$th ordinal such that $L_\beta\models T$. It is clear that such a function $f$ serves as the desired reduction.
  
  In order to define such an $f$ in an absolutely $\bm{\Delta}^1_2$ way, we will use the infinite time Turing machine model. Briefly, an infinite time Turing machine is an extension of the classical Turing machine, with finitely many states, and tapes for input, output, scratch, and an oracle. At stage $\omega$ the machine is not considered to have diverged but continues running. In fact at any limit stage, the machine is put in a special limit state, the tape pointers are reset to the left, and the tape cells are replaced with the limit superior of their values so far. We refer the reader to \cite{hamkins-ittm} for other background on infinite time Turing computation. By \cite[Theorem~2.6]{coskey-ittm}, any function which may be computed by an oracle infinite time Turing machine is absolutely $\bm{\Delta}^1_2$.
  
  Let $M$ be an infinite time Turing machine which runs as follows. Let $x$ be a given input and assume $x$ is a code for an ordinal $\alpha$. The machine $M$ will recursively construct for each $i\leq\alpha$ a code $y_i$ for an ordinal $\beta_i$. If the $y_i$ have been constructed for $i<j$, construct a code $z$ for $\sup_{i<j}\beta_i$ and evaluate $g(z)$ ($M$ can evaluate a continuous function by \cite[Theorem~2.1]{coskey-ittm}). For each ordinal $\beta$ between $\sup_{i<j}\beta_i$ and the value of $g(z)$, $M$ constructs a code for $L_\beta$ ($M$ can construct such a code by \cite[Theorem~7]{hamkins-itmt}). Furthermore $M$ checks whether $L_\beta\models T$ ($M$ can evaluate arithmetic expressions by \cite[Theorem~2.1]{hamkins-ittm}). By the construction of $g$, the answer is guaranteed to be Yes for some $\beta$, and the first time this happens we let $y_j=$ the code for that $\beta$. When the final code $y_\alpha$ for $\beta_\alpha$ has been calculated, $M$ outputs a code for $L_{\beta_\alpha}$. The construction guarantees that the output is a code for $L_\beta$ where $\beta$ is the $\alpha$th ordinal such that $L_\beta\models T$, as desired.
\end{proof}

The next result uses the above lower bound to provide a further consequence for the complexity of the classification of well-founded models of $T$.

\begin{thm}
  Assume $0^\sharp$ exists, and let $T=\Th(L_{\omega_1})$. Then $\cong_{\WFT}$ is not absolutely $\bm{\Delta}^1_2$ reducible to any Borel equivalence relation $E$.
\end{thm}

\begin{proof}
  By the previous theorem it is sufficient to show that there is no absolutely $\bm{\Delta}^1_2$ reduction from $E_{\omega_1}$ to a Borel equivalence relation. Indeed, if there were such a reduction $f$, then it would be possible to find an absolutely $\bm{\Delta}^1_2$ injection $F$ from codes for ordinals to codes for sets of reals of bounded Borel rank. (In fact one can take $F(x)$ to be a code for $[f(x)]_E$.) However, this contradicts the remark in the last paragraph of Section~3 of \cite{hjorth}, which states that no such mapping exists.
\end{proof}

\bibliographystyle{alpha}
\bibliography{bib}

\end{document}